\documentclass[12pt, reqno]{amsart}
\usepackage{amsmath, amsthm, amscd, amsfonts, amssymb, graphicx, color}
\usepackage[bookmarksnumbered, colorlinks, plainpages]{hyperref}
\usepackage{blkarray}
\usepackage{physics}
\usepackage{amsmath}
\usepackage{mathdots}
\hypersetup{colorlinks=true,linkcolor=red, anchorcolor=green, citecolor=cyan, urlcolor=red, filecolor=magenta, pdftoolbar=true}

\renewcommand{\phi}{\varphi}
\newtheorem{theorem}{Theorem}[section]
\newtheorem{lemma}[theorem]{Lemma}

\newtheorem{corollary}[theorem]{Corollary}
\theoremstyle{definition}

\newtheorem{example}[theorem]{Example}

\theoremstyle{remark}
\newtheorem{remark}[theorem]{Remark}

\numberwithin{equation}{section}

\begin{document}

\setcounter{page}{1}

\title[ Product of rectangular Toeplitz (Hankel) Matrices.]{ Product of rectangular Toeplitz (Hankel) Matrices.}

\author[S. Bensliman, A. Yagoub and K. Toumache ]{Said Bensliman $^1$ $^{*}$, Ameur Yagoub $^1$, Kamel Toumache $^2$.}

\address{$^{1}$ Laboratoire de math\'ematiques pures et appliqu\'es. Universit\'e de Amar Telidji. Laghouat, 03000. Algeria.}
\email{\textcolor[rgb]{0.00,0.00,0.84}{s.bensliman@lagh-univ.dz}}
\email{\textcolor[rgb]{0.00,0.00,0.84}{a.yagoub@lagh-univ.dz}}
\address{$^{2}$ Laboratory of Mathematics and its Applications (LaMA), University of Medea.}
\email{\textcolor[rgb]{0.00,0.00,0.84}{kamel\_toumache@yahoo.fr}}



\date{Received: xxxxxx, Accepted: zzzzzz.
\newline \indent $^{*}$ Corresponding author: Said BENSLIMAN}
\begin{abstract}In this paper, we study products of asymmetric Toeplitz matrices, we give necessary and sufficient conditions for the product of two asymmetric Toeplitz matrices compatible sizes is asymmetric Toeplitz matrix. We also give some results related to the isometry.

\end{abstract} \maketitle
\section{Introduction.}
The product of matrices is an important technique in numerical calculations. 
Some authors have deduced several properties of asymmetric Toeplitz (or Hankel) matrices based on the operators  known  as asymmetric truncated Toeplitz operators, such as \cite{J,J2,R}. When talking about asymmetric Toeplitz operators, they are introduced by \cite{C,C2}, the characterization of product of  asymmetric truncated Toeplitz operators has been discussed in \cite{YA}. In \cite{L}, Lim and Ye show that every matrix is a product of Toeplitz matrices. More interesting results about Toeplitz and Hankel matrices can be found in \cite{C3,G2,G,GU,H,I,I2,P,Z}. The product of two arbitrary Toeplitz (respectively, Hankel) matrices is not necessarily a Toeplitz (respectively, Hankel) matrix.

The main purpose of this paper is another type of product, which is the product of two asymmetric Toeplitz matrices $A$ and $B$, and we find the conditions for them until the product $AB$ is an asymmetric Toeplitz matrix. Since our goal is to find the conditions for A and B, some assumptions about the coefficients of both $A$ and $B$ are needed. Finally, we characterize isometric asymmetric Toeplitz matrices.

This paper is organized as follows: In Section 2, we present some preliminary knowledge of asymmetric  Toeplitz matrices and the technique to know these matrices. In Section 3, we give several results for the  product of asymmetric Toeplitz matrices. In Section 4, we present some results on isometry of asymmetric  Toeplitz matrices.

\section{ Preliminaries.}
 Let $M_{n\times m}(\mathbb{C)}$ be the set of complex $n \times m$ matrices. In this paper, we will number the rows and columns of $n\times m$ matrices from $0$ to $n-1$ and $0$ to $m-1$, respectively. Thus, we will utilize the following notation, $\lbrace e_{1},e_{2},...,e_{n}\rbrace$, $\lbrace\varepsilon_{1},\varepsilon_{2},...,\varepsilon_{m} \rbrace$, and $\lbrace\zeta_{1},\zeta_{2},...,\zeta_{l}\rbrace$ for the standard basis vectors of $\mathbb{C}^{n}$,  $\mathbb{C}^{m}$, and  $\mathbb{C}^{l}$, respectively. Let $I_{n\times m}=(r_{ij})$ and $J_{n\times m}=(l_{ij})$ be $n\times m $ matrices defined as 
\begin{center}
$r_{ij}=\begin{cases}
1\quad if \quad i=j
\\
0\quad if \quad i\neq j
\end{cases}$
, and 
$l_{ij}=\begin{cases}
1\quad if \quad i+j=m+1
\\
0\quad otherwise
\end{cases}$.
\end{center}
If $n=m$, then $I_{n}=I_{n\times n}$ and $J_{n}=J_{n\times n}$.

Let $S_{n}$ and $P_{n}$ be $n\times n$ matrices defined as
\begin{center}
$S_{n}=\left[ \begin{array}{ccccc}
0 & 0 &  \cdots &  \\
1 & 0 &   &  \\
0 & 1 &   &  \\
\vdots &  &  \ddots  &  \\
0  & \cdots &   & 1 & 0
\end{array} \right], \textit{ and }P_{n}=\left[ \begin{array}{cccccc}
0 & &  \cdots & 0 & 1 \\
\vdots  & &  & 1& 0 \\
 &  & \iddots & \ & \vdots \\
 &  &  &   &  \\
1  &  &  &   & 0 
\end{array} \right]. $
\end{center}
Let $\Delta A=A-S_{n}AS_{m}^{\ast}$ be the displacement of matrix $A\in M_{n\times m}(\mathbb{C)}$.
 Recall that an asymmetric Toeplitz matrix is a matrix of the form 
\begin{center}
$A=\left[\begin{array}{ccccccc}
a_{0}&  \overline{\alpha}_{1} &  \cdots & &  &\overline{\alpha}_{m-1} \\
a_{1} & a_{0} &  & & & \vdots \\
\vdots &\ddots &  & & &  \\
   &  &  & &  &  \\
a_{n-1}&  & &  &  & 
\end{array}\right].$
\end{center}
The matrix $H$ is called an asymmetric Hankel matrix if it has the form  
\begin{center}
$H=\left[\begin{array}{ccccccc}
a_{1}&  a_{2} & & \cdots & a_{m-1}  & a_{0} \\
a_{2} & a_{3} & & & & \overline{\alpha}_{1} \\
\vdots & &  &  & \iddots &  \vdots \\
   &  &  & & &  \\
a_{n-1}&  & & &  & \overline{\alpha}_{n-1}
\end{array}\right].$
\end{center}
While $H$ will be written in the form $H = AP_{m}$, where $A$ is a Toeplitz matrix, note that every Hankel matrix can be obtained from an appropriate Toeplitz matrix. If $n=m$, then $A$ is a Toeplitz matrix and $H$ is a Hankel matrix. From the definition above, it is evident that if $A$ is an asymmetric Toeplitz (respectively, Hankel) matrix, then $A^{T}$ and $A^{\ast}$ are also asymmetric Toeplitz (respectively, Hankel) matrices.

Let $a$ and $\alpha$ be vectors in $\mathbb{C}^{n}$ and $\mathbb{C}^{m}$, respectively, where 
$a=(0,a_{1},a_{2},...,a_{n-1})^{T}$ and $\alpha=(0,\alpha_{1},\alpha_{2},...,\alpha_{m-1})^{T}$. We denote $\hat{a}=(0,\overline{a}_{n-1},\overline{a}_{n-2},...,\overline{a}_{1})^{T}$. And let $A(a,\alpha)$ denote an asymmetric Toeplitz matrix, where
\begin{center} 
$A(a,\alpha)=\left[\begin{array}{ccccccc}
0 &  \overline{\alpha}_{1} &  \cdots &  &\overline{\alpha}_{m-1} \\
a_{1} & 0 &  & & \vdots \\
\vdots &\ddots &  &  &  \\
   &  &  &  &  \\
a_{n-1}&  &  &  & 
\end{array}\right],$
\end{center}
and let $H(a,\alpha)$ denote an asymmetric Hankel matrix, where
\begin{center} 
$H(a,\alpha)=\left[\begin{array}{ccccccc}
a_{1}&  a_{2} & & \cdots & a_{m-1} & 0 \\
a_{2} & a_{3} & & & & \overline{\alpha}_{1} \\
\vdots & &  &  & \iddots & \vdots \\
   &  &  & & &  \\
&  &  & & & \overline{\alpha}_{n-1}
\end{array}\right] . $
\end{center}
Observe that every asymmetric Toeplitz matrix $A$ can be written as 
$$A=A_{0}+a_{0}I_{n\times m},$$ where $A_{0}=A(a,\alpha).$ And every asymmetric Hankel matrix $H$ can be written as 
$$H=H_{0}+a_{0}J_{n\times m},$$
where $H_{0}=H(a,\alpha )$. We have $A^{\ast}(a,\alpha)=A(\alpha,a)$.

In this article, the set of $n\times m$ asymmetric Toeplitz (respectively, Hankel) matrices is denoted by $\mathcal{T}(n,m)$ (respectively, $\mathcal{H}(n,m)$), and the set of $n\times m$ asymmetric Toeplitz (respectively, Hankel) matrices for $a_{0}=0$ is denoted by $\mathcal{T}_{0}(n,m)$ (respectively, $\mathcal{H}_{0}(n,m)$).
\begin{lemma}
Let $A\in\mathcal{T}(n,m)$. Then 
$$A=\sum\limits_{i=0}^{\min(n,m)-1}S_{n}^{i}(\Delta A)S_{m}^{i\ast}.$$
\end{lemma}
\begin{proof}
Since $S_{n}^{n}=0$ and $S_{m}^{m\ast}=0$, we have
\begin{align*}
\sum\limits_{i=0}^{\min(n,m)-1}S_{n}^{i}(\Delta A)S_{m}^{i\ast}
&=\sum\limits_{i=0}^{\min(n,m)-1}S_{n}^{i}AS_{m}^{i\ast}-S_{n}^{i+1}AS_{m}^{i+1\ast}
\\
&=A-S_{n}^{\min(n,m)}AS_{m}^{\min(n,m)\ast}=A.
\end{align*}
\end{proof}
For two vectors $x$ and $y$ belong to $\mathbb{C}^{n}$ and $\mathbb{C}^{m}$, respectively. $x\otimes y$ denotes a tensor product of $x$ and $y$, it is defined as 
\begin{center}
$(x\otimes y)z=\langle z,y\rangle x\qquad$
for all $z\in \mathbb{C}^{m},$
\end{center}
where $x\otimes y=x\overline{y}^{T}$, and $\langle z,y\rangle =z_{1}\overline{y}_{1}+z_{1}\overline{y}_{2}+...+z_{m}\overline{y}_{m}.$
It is worth mentioning the following simple lemma, which determines when the
matrix in $ M_{n\times m}(\mathbb{C)}$ is an asymmetric Toeplitz matrix.
\begin{lemma}\label{Lemma 2.1}
Let $A$ be $n\times m$ matrix. Then $A\in\mathcal{T}(n,m)$ if and only if  
$$\Delta A=u\otimes \varepsilon_{0}+e_{0}\otimes v,$$
for some $u\in\mathbb{C}^{n}$ and $v\in\mathbb{C}^{m}$.
\end{lemma}
\begin{proof}
Let $A$ be $n\times m$ asymmetric Toeplitz matrix, then the matrix $S_{n}AS_{m}^{\ast}$ is the same as matrix $A$ except for the first row and the first column of the matrix $S_{n}AS_{m}^{\ast}$, which are zero. Thus, $A-S_{n}AS_{m}^{\ast}$ is $n\times m$ matrix all of its coefficients are zero except those in the first row and the first column.
\end{proof}
\begin{remark}
$\Delta (A(a,\alpha))=a\otimes \varepsilon_{0}+e_{0}\otimes \alpha.$
\end{remark}
 Throughout this paper, we assume that $A=A_{0}+a_{0}I_{n\times m}$, and $B=B_{0}+b_{0}I_{m\times l}$, where  $A_{0}=A(a,\alpha)$, and $B_{0}=A(b,\beta)$, for $a=(0,a_{1},a_{2},...,a_{n-1})^{T}$, $\alpha=(0,\alpha_{1},\alpha_{2},...,\alpha_{m-1})^{T}$, $b=(0,b_{1},b_{2},...,b_{m-1})^{T}$, and $\beta=(0,\beta_{1},\beta_{2},...,\beta_{l-1})^{T}$.
\begin{lemma}\label{Lemma 2.3}
Let $A\in\mathcal{T}(n,m)$ and $B\in\mathcal{T}(m,l)$. Then we have 
\begin{enumerate}
\item
$S_{n}A_{0}\varepsilon_{m-1}= \hat{\alpha}_{(m,n)},$ where $\hat{\alpha}_{(m,n)}=(0,\overline{\alpha}_{m-1},...,\overline{\alpha}_{m-n+1})^{T}$ if $n\leq m,$ and $\hat{\alpha}_{(m,n)}=(0,\overline{\alpha}_{m-1},...,\overline{\alpha}_{1},0,a_{1},...,a_{n-m-1})^{T}$ if $m < n.$
\item
$S_{l}B_{0}^{\ast}\varepsilon_{m-1}=\hat{b}_{(m,l)}.$ where $\hat{b}_{(m,l)}=(0,\overline{b}_{m-1},...,\overline{b}_{m-l+1})^{T}$ if $l\leq m,$ and $\hat{b}_{(m,l)}=(0,\overline{b}_{m-1},...,\overline{b}_{1},0,\beta_{1},...,\beta_{l-m-1})^{T}$ if $m < l.$
\item
$A_0\varepsilon_{0}=a.$
\item
$I_{n\times m}\varepsilon_{0}=e_{0},$ and $ I_{l\times m}\varepsilon_{0}=\zeta_{0}.$
\item
$I_{m\times n}\hspace*{1mm} a=a_{(n,m)}^{\sharp}$, where $a_{(n,m)}^{\sharp}=(0,a_{1},a_{2},..,a_{m-1})^{T}$ if $m\leq n,$ and
$a_{(n,m)}^{\sharp}=(0,a_{1},a_{2},...,a_{n-1},0,...,0)^{T}\in\mathbb{C}^{m}$ if $n\leq m.$
\item
If $n\leq m$, then $S_{n}I_{n\times m}\varepsilon_{m-1}=0$. And if $m< n$, then $S_{n}I_{n\times m}\varepsilon_{m-1}=e_{m}$.
\item
If $n\leq m$, then $S_{n}I_{n\times m}=I_{n\times m}S_{m}$. And if $m< n$, then $S_{n}I_{n\times m}=I_{n\times m}S_{m}+e_{m}\otimes \varepsilon_{m-1}$.
\end{enumerate} 
\end{lemma}
\begin{proof}
\begin{enumerate}
\item
If $n\leq m$, we have 
$$A_{0}\varepsilon_{m-1}=\left[\begin{array}{ccccccc}
0 &  \overline{\alpha}_{1} &  \cdots &  &\overline{\alpha}_{m-1} \\
a_{1} & 0 &  & & \vdots \\
\vdots &\ddots &  &  &  \\
   &  &  &  &  \\
a_{n-1}&  &  &  & 
\end{array}\right]\left[\begin{array}{c}
0 \\
0 \\
\vdots  \\
    \\
    \\
1
\end{array}\right] =\left[\begin{array}{c}
\overline{\alpha}_{m-1} \\
 \vdots \\
  \\
    \\
\overline{\alpha}_{m-n}
\end{array}\right],$$
then 
 $S_{n}A_{0}\varepsilon_{m-1}=(0,\overline{\alpha}_{m-1},...,\overline{\alpha}_{m-n+1})^{T}=\hat{\alpha}_{(m,n)}.$ 
 If $m < n$, we have 
$$A_{0}\varepsilon_{m-1}=\left[\begin{array}{ccccccc}
0 &  \overline{\alpha}_{1} &  \cdots &  &\overline{\alpha}_{m-1} \\
a_{1} & 0 &  & & \vdots \\
\vdots &\ddots &  &  &  \\
   &  &  &  &  \\
a_{n-1}&  &  &  & 
\end{array}\right]\left[\begin{array}{c}
0 \\
0 \\
\vdots  \\
    \\
    \\
1
\end{array}\right] =\left[\begin{array}{c}
\overline{\alpha}_{m-1} \\
 \vdots \\
 \overline{\alpha}_{1} \\
  0  \\
  a_{1}\\
  \vdots\\
a_{n-m}
\end{array}\right],$$
then 
 $S_{n}A_{0}\varepsilon_{m-1}=(0,\overline{\alpha}_{m-1},...,\overline{\alpha}_{1},0,a_{1},...,a_{n-m-1})^{T}=\hat{\alpha}_{(m,n)}.$
 \item
The proof is the same as proof (1).
\item
$A_{0}\varepsilon_{0}=\left[\begin{array}{ccccccc}
0 &  \overline{\alpha}_{1} &  \cdots &  &\overline{\alpha}_{m-1} \\
a_{1} & 0 &  & & \vdots \\
\vdots &\ddots &  &  &  \\
   &  &  &  &  \\
a_{n-1}&  &  &  & 
\end{array}\right]\left[\begin{array}{c}
1 \\
0 \\
\vdots  \\
    \\
    \\
0
\end{array}\right] =\left[\begin{array}{c}
0 \\
a_{1}\\
 \vdots \\
  \\
    \\
a_{n-1}
\end{array}\right]=a.$
\item
We have
$$
I_{n\times m}\varepsilon_{0}=\sum\limits_{i=0}^{ \min(n,m)-1}(e_{i}\otimes\varepsilon_{i})\varepsilon_{0}=\sum\limits_{i=0}^{\min(n,m)-1}\langle\varepsilon_{0},\varepsilon_{i}\rangle e_{i}=e_{0},$$
and
 $$ I_{l\times m}\varepsilon_{0}=\sum\limits_{i=0}^{\min(l,m)-1}( \zeta_{i}\otimes\varepsilon_{i})\varepsilon_{0}=\sum\limits_{i=0}^{\min(l,m)-1}\langle\varepsilon_{0},\varepsilon_{i}\rangle \zeta_{i}=\zeta_{0}.$$
\item 
Since $\langle a,e_{0}\rangle =0$, then we have
$$I_{m\times n} \hspace*{1mm}a=\sum\limits_{i=0}^{\min(n,m)-1}(\varepsilon_{i}\otimes e_{i})a=\sum\limits_{i=0}^{\min(n,m)-1} \langle a,e_{i}\rangle \varepsilon_{i}=\sum\limits_{i=1}^{\min(n,m)-1} a_{i}\varepsilon_{i}=a_{(n,m)}^{\sharp}.$$
\item We have
$$
S_{n}I_{n\times m}\varepsilon_{m-1}=S_{n}\sum\limits_{i=0}^{\min(n,m)-1}(e_{i}\otimes\varepsilon_{i})\varepsilon_{m-1}=\sum\limits_{i=0}^{\min(n,m)-1}(e_{i+1}\otimes\varepsilon_{i})\varepsilon_{m-1}=\sum\limits_{i=0}^{\min(n,m)-1}\langle\varepsilon_{m-1},\varepsilon_{i}\rangle e_{i+1}.$$
Thus, $S_{n}I_{n\times m}\varepsilon_{m-1}=0$ if $n\leq m$, and $S_{n}I_{n\times m}\varepsilon_{m-1}=e_{m}$ if $m< n$.
\item
We have 
$$\Delta(I_{n\times m})=I_{n\times m}-S_{n}I_{n\times m}S_{m}^{\ast}=e_{0}\otimes \varepsilon_{0},$$
then 
$$S_{n}I_{n\times m}S_{m}^{\ast}=I_{n\times m}-e_{0}\otimes \varepsilon_{0},$$
since $I_{m}=S_{m}^{\ast}S_{m}+\varepsilon_{m-1}\otimes\varepsilon_{m-1}$ and $S_{m}^{\ast}\varepsilon_{0}=0$, then 
$$S_{n}I_{n\times m}S_{m}^{\ast}S_{m}=I_{n\times m}S_{m},$$
then
$$S_{n}I_{n\times m}(I_{m}-\varepsilon_{m-1}\otimes\varepsilon_{m-1})=I_{n\times m}S_{m},$$
hence
$$S_{n}I_{n\times m}-S_{n}I_{n\times m}\varepsilon_{m-1}\otimes\varepsilon_{m-1}=I_{n\times m}S_{m},$$
we use (6) to complete this proof.  
\end{enumerate}
\end{proof}
To proceed with our characterization, we need several auxiliary lemmas. The first of them is a well-known fact.
\begin{lemma}\label{L1}
Let $x,x'\in\mathbb{C}^{n}$ and $y,y'\in\mathbb{C}^{m}$. If the rank-one matrices $M=x\otimes y$ and $N=x'\otimes y'$ are identical, then one
of the following three cases holds
\begin{enumerate}
\item $x=x'=0$ or $y=y'=0$.
\item $x=y'=0$ or $y=x'=0$.
\item $x= \lambda x' $ and $\overline{\lambda}y=y'$ for some nonzero scalar $\lambda$.
\end{enumerate}
By convention $\frac{1}{\infty}=0$, we can propose that $x\otimes y=x'\otimes y'$, then one of the following two cases holds
\begin{enumerate}
\item[i.] Trivial case: $x=x'=0$ or $y=y'=0$.
\item[ii.] $x= \lambda x' $ and $\overline{\lambda}y=y'$ for some scalar $\lambda\in\hat{\mathbb{C}}=\mathbb{C}\cup\lbrace \infty\rbrace$.
\end{enumerate}
\end{lemma}
The subsequent lemma serves as a crucial element in our work, enabling the manipulation of the product of asymmetric Toeplitz matrices.
\begin{lemma}\label{Lemma 2.4}
Let $A\in\mathcal{T}(n,m)$ and $B\in\mathcal{T}(m,l)$. Then we have 

$\Delta (AB)=a\otimes\beta-\hat{\alpha}_{(m,n)}\otimes \hat{b}_{(m,l)}+A_{0}b\otimes \zeta_{0}+e_{0}\otimes S_{l}B_{0}^{\ast}S_{m}^{\ast}\alpha +a_{0}b_{0}\Delta (I_{n\times m}I_{m\times l})
+b_{0}\Delta (A_{0}I_{m\times l})+a_{0}\Delta (I_{n\times m}B_{0})
.$
\end{lemma}
\begin{proof}
\begin{align}\label{2.0.1}
\Delta(AB)&=\Delta\Big[( A_{0}+a_{0}I_{n\times m})(B_{0}+b_{0}I_{m\times l})\Big] \nonumber \\
&=\Delta (A_{0}B_{0}+a_{0}I_{n\times m}B_{0}+b_{0}A_{0}I_{m\times l}+a_{0}b_{0}I_{n\times m}I_{m\times  l}) \nonumber \\
&=\Delta (A_{0}B_{0})+a_{0}\Delta (I_{n\times m}B_{0})+b_{0}\Delta (A_{0}I_{m\times l})+a_{0}b_{0}\Delta( I_{n\times m}I_{m\times l}),
\end{align}
on the other hand, by Lemma \ref{Lemma 2.3}, we have 
\begin{align}\label{2.0.2}
\Delta (A_{0}B_{0})&=A_{0}B_{0}-S_{n}A_{0}B_{0}S_{l}^{\ast} \nonumber \\
&=A_{0}B_{0}-A_{0}S_{m}B_{0}S_{l}^{\ast}+A_{0}S_{m}B_{0}S_{l}^{\ast}-S_{n}A_{0}B_{0}S_{l}^{\ast}\nonumber \\
&=A_{0}(B_{0}-S_{m}B_{0}S_{l}^{\ast})+A_{0}S_{m}B_{0}S_{l}^{\ast}-S_{n}A_{0}(S_{m}^{\ast}S_{m}+\varepsilon_{m-1}\otimes\varepsilon_{m-1}) B_{0}S_{l}^{\ast} \nonumber \\
&=A_{0}\Delta B_{0}+\Delta A_{0}(S_{m}B_{0}S_{l}^{\ast})-S_{n}A_{0}(\varepsilon_{m-1}\otimes\varepsilon_{m-1}) B_{0}S_{l}^{\ast} \nonumber\\
&=A_{0}(b\otimes \zeta_{0}+\varepsilon_{0}\otimes\beta)+(a\otimes \varepsilon_{0}+e_{0}\otimes\alpha)(S_{m}B_{0}S_{l}^{\ast})-S_{n}A_{0}\varepsilon_{m-1}\otimes S_{l}B_{0}^{\ast}\varepsilon_{m-1} \nonumber\\
&=A_{0}b\otimes \zeta_{0}+A_{0}\varepsilon_{0}\otimes\beta+a\otimes S_{l}B_{0}^{\ast}S_{m}^{\ast}\varepsilon_{0}+e_{0}\otimes S_{l}B_{0}^{\ast}S_{m}^{\ast}\alpha -S_{n}A_{0}\varepsilon_{m-1}\otimes S_{l}B_{0}^{\ast}\varepsilon_{m-1} \nonumber \\
&=A_{0}b\otimes \zeta_{0}+a\otimes\beta+e_{0}\otimes S_{l}B_{0}^{\ast}S_{m}^{\ast}\alpha -\hat{\alpha}_{(m,n)}\otimes \hat{b}_{(m,l)}.
\end{align}
Combining (\ref{2.0.1}) and (\ref{2.0.2}), we get
\begin{align*}
\Delta (AB)&=a\otimes\beta-\hat{\alpha}_{(m,n)}\otimes \hat{b}_{(m,l)}+A_{0}b\otimes \zeta_{0}+e_{0}\otimes S_{l}B_{0}^{\ast}S_{m}^{\ast}\alpha +a_{0}b_{0}\Delta (I_{n\times m}I_{m\times l})\\
&+b_{0}\Delta (A_{0}I_{m\times l})+a_{0}\Delta (I_{n\times m}B_{0}).
\end{align*}
\end{proof}
\section{product of two matrices in $\mathcal{T}(n,m)$. }
In this section, we give several results for the product of asymmetric Toeplitz matrices. Before going into the  main results, we need the following lemma

\begin{lemma}\label{Lemma 2.5}
Let $A\in\mathcal{T}(n,m)$ and $B\in\mathcal{T}(m,l)$. Then we have 
\begin{enumerate}
\item
If $n\leq m$, then $I_{n\times m}B_{0}\in\mathcal{T}(n,l)$. And if $m< n$, then 
$$\Delta (I_{n\times m}B_{0})=b_{(m,n)}^{\sharp}\otimes \zeta_{0}+e_{0}\otimes \beta -e_{m}\otimes\hat{b}_{(m,l)}.$$
\item
If $l\leq m$, then $A_{0}I_{m\times l}\in\mathcal{T}(n,l)$. And if $m< l$, then
$$\Delta (A_{0}I_{m\times l})=a\otimes \zeta_{0}+e_{0}\otimes \alpha_{(m,l)}^{\sharp}-\hat{\alpha}_{(m,n)}\otimes \zeta_{m}.$$ 
\item
If $n\leq m$ or $l \leq m$, then $I_{n\times m}I_{m\times l}\in\mathcal{T}(n,l)$. And if $m< \min(n,l)$, then $$\Delta (I_{n\times m}I_{m\times l})=e_{0}\otimes \zeta_{0}-e_{m}\otimes \zeta_{m}.$$
\end{enumerate}
\end{lemma}

\begin{proof}
\begin{enumerate}
\item
We have
\begin{align*}
\Delta (I_{n\times m}B_{0})&=I_{n\times m}B_{0}-S_{n}I_{n\times m}B_{0}S_{l}^{\ast}=I_{n\times m}(\Delta B_{0}+S_{m}B_{0}S_{l}^{\ast})-S_{n}I_{n\times m}B_{0}S_{l}^{\ast}\\
&=I_{n\times m}\Delta B_{0}+(I_{n\times m}S_{m}-S_{n}I_{n\times m})B_{0}S_{l}^{\ast}.
\end{align*}
By Lemma \ref{Lemma 2.3}, if $n\leq m$, then 
$$
\Delta (I_{n\times m}B_{0})=I_{n\times m}\Delta B_{0}=I_{n\times m}(b\otimes\zeta_{0}+\varepsilon_{0}\otimes\beta)=b_{(m,n)}^{\sharp}\otimes \zeta_{0}+e_{0}\otimes \beta,
$$
then by Lemma \ref{Lemma 2.1}, $I_{n\times m}B_{0}\in\mathcal{T}(n,l)$. And if $m< n$, then
$$
\Delta (I_{n\times m}B_{0})=I_{n\times m}\Delta B_{0}-(e_{m}\otimes \varepsilon_{m-1})B_{0}S_{l}^{\ast}=b_{(m,n)}^{\sharp}\otimes \zeta_{0}+e_{0}\otimes \beta -e_{m}\otimes\hat{b}_{(m,l)}.
$$
\item We have
\begin{align*}
\Delta (A_{0}I_{m\times l})&= A_{0}I_{m\times l}-S_{n}A_{0}I_{m\times l}S_{l}^{\ast}=(\Delta A_{0}+S_{n}A_{0}S_{m}^{\ast})I_{m\times l}-S_{n}A_{0}I_{m\times l}S_{l}^{\ast}\\
&=(\Delta A_{0})I_{m\times l}+S_{n}A_{0}(S_{m}^{\ast}I_{m\times l}-I_{m\times l}S_{l}^{\ast})=(\Delta A_{0})I_{m\times l}+S_{n}A_{0}(I_{l\times m}S_{m}-S_{l}I_{l\times m})^{\ast}.
\end{align*}
By Lemma \ref{Lemma 2.3}, if $l\leq m$, then
$$\Delta (A_{0}I_{m\times l})=(\Delta A_{0})I_{m\times l}=(a\otimes \varepsilon_{0}+e_{0}\otimes \alpha)I_{m\times l}=a\otimes \zeta_{0}+e_{0}\otimes \alpha_{(m,l)}^{\sharp},$$
 then by Lemma \ref{Lemma 2.1}, $A_{0}I_{m\times l}\in\mathcal{T}(n,l)$. And if $m< l$, then 
$$\Delta (A_{0}I_{m\times l})=(\Delta A_{0})I_{m\times l}-S_{n}A_{0}(\varepsilon_{m-1}\otimes \zeta_{m})=a\otimes \zeta_{0}+e_{0}\otimes \alpha_{(m,l)}^{\sharp}-\hat{\alpha}_{(m,n)}\otimes \zeta_{m}.$$

\item
If $n\leq m$, we have
\begin{eqnarray*}
I_{n\times m}I_{m\times l}&=&\sum\limits_{i=0}^{n-1}(e_{i}\otimes \varepsilon_{i})\sum\limits_{i=0}^{\min{(m,l)}-1}(\varepsilon_{i}\otimes \zeta_{i})=\sum\limits_{i=0}^{n-1}\sum\limits_{k=0}^{\min{(m,l)}-1}(e_{i}\otimes \varepsilon_{i})(\varepsilon_{k}\otimes \zeta_{k})\\
&=&\sum\limits_{i=0}^{n-1}\sum\limits_{k=0}^{\min{(m,l)}-1}\langle \varepsilon_{k},\varepsilon_{i} \rangle (e_{i}\otimes \zeta_{k})=\sum\limits_{i=0}^{\min{(n,l)}-1}e_{i}\otimes \zeta_{i}=I_{n\times l}\in\mathcal{T}(n,l).
\end{eqnarray*}	
And if $l \leq m$, we have
\begin{eqnarray*}
I_{n\times m}I_{m\times l}&=&\sum\limits_{i=0}^{\min{(n,m)}-1}(e_{i}\otimes \varepsilon_{i})\sum\limits_{i=0}^{l-1}(\varepsilon_{i}\otimes \zeta_{i})=\sum\limits_{i=0}^{\min{(n,m)}-1}\sum\limits_{k=0}^{l-1}(e_{i}\otimes \varepsilon_{i})(\varepsilon_{k}\otimes \zeta_{k})\\
&=&\sum\limits_{i=0}^{\min{(n,m)}-1}\sum\limits_{k=0}^{l-1}\langle \varepsilon_{k},\varepsilon_{i} \rangle (e_{i}\otimes \zeta_{k})=\sum\limits_{i=0}^{\min{(n,l)}-1}e_{i}\otimes \zeta_{i}=I_{n\times l}\in\mathcal{T}(n,l).
\end{eqnarray*}	
And if $m <\min(n,l)$,  we have $I_{n\times m}I_{m\times l}=\sum\limits_{i=0}^{m-1}e_{i}\otimes\zeta_{i}$, then
 \begin{align*}
\Delta (I_{n\times m}I_{m\times l})&=\sum\limits_{i=0}^{m-1}e_{i}\otimes\zeta_{i}-S_{n}\Big(\sum\limits_{i=0}^{m-1}e_{i}\otimes\zeta_{i}\Big)S_{l}^{\ast}=\sum\limits_{i=0}^{m-1}e_{i}\otimes\zeta_{i}-\sum\limits_{i=0}^{m-1}e_{i+1}\otimes \zeta_{i+1}\\
&=\sum\limits_{i=0}^{m-1}e_{i}\otimes\zeta_{i}-\sum\limits_{i=1}^{m}e_{i}\otimes\zeta_{i}=e_{0}\otimes \zeta_{0}-e_{m}\otimes \zeta_{m}.
\end{align*}
\end{enumerate}
\end{proof}

\begin{theorem}\label{coro 3.3}
Let $A\in\mathcal{T}(n,m)(A\neq 0)$ and $B\in\mathcal{T}(m,l)(B\neq 0)$. If$\hspace*{2mm}\max(n,l)\leq m$, then $AB\in\mathcal{T}(n,l)$ if and only if one of the following conditions is met
\begin{enumerate}
\item
$A$ has a form  
\begin{equation}\label{MATRIX A1}
\left[\begin{array}{ccccccc}
a_{0} & \overline{\alpha}_{1} & \cdots  &   \overline{\alpha}_{m-n} & 0 &\cdots & 0\\
0 & \ddots  &   & \\
 \vdots & &  & \\
0 & &   &  
\end{array}\right] 
\end{equation}
or $B$ has a form
\begin{equation}\label{MATRIX B1}
\left[\begin{array}{ccccccc}
b_{0} & 0 & \cdots  &   0 \\
b_{1} & \ddots  &   & \\
 \vdots & &  & \\
b_{m-l} & &   & \\
  0 & &   &  \\
 \vdots & &   &  \\
 0 & &   &  
\end{array}\right].
\end{equation}

\item
$A=A(\dfrac{1}{\lambda} \hat{\alpha}_{(m,n)},\alpha)+a_{0}I_{n\times m}$ and $B=A(b,\overline{\lambda}\hat{b}_{(m,l)})+b_{0}I_{m\times l}$, for some $\lambda\in\hat{\mathbb{C}}$.
\end{enumerate}
\end{theorem}

\begin{proof}
 If $\max (n,l)\leq m$. By Lemma \ref{Lemma 2.4} and Lemma \ref{Lemma 2.5}, we have
 $$\Delta(AB)=a\otimes\beta-\hat{\alpha}_{(m,n)}\otimes \hat{b}_{(m,l)}+\gamma_{1}\otimes \zeta_{0}+e_{0}\otimes \gamma_{2},$$

with $\gamma_{1}=A_{0}b+a_{0}b_{(m,n)}^{\sharp}+ab_{0}+a_{0}b_{0}e_{0},$ and $\gamma_{2}=S_{l}B_{0}^{\ast}S_{m}^{\ast}\alpha +\overline{a}_{0}\beta +\overline{b}_{0}\alpha_{(m,l)}^{\sharp}$.

 Since $a\otimes\beta$ and $\hat{\alpha}_{(m,n)}\otimes \hat{b}_{(m,l)}$ are $n\times l$ matrices, where the first row and the first column are zero, then $AB\in\mathcal{T}(n,l)$ if and only if 
$$a\otimes\beta-\hat{\alpha}_{(m,n)}\otimes \hat{b}_{(m,l)}=0.$$
By Lemma \ref{L1}, we conclude that $AB\in\mathcal{T}(n,l)$ if and only if one of the following conditions is met
\begin{enumerate}
\item $a=\hat{\alpha}_{(m,n)}=0$, then $A$ has a form (\ref{MATRIX A1}). Or $\beta=\hat{b}_{(m,l)}=0$, then $B$ has a form (\ref{MATRIX B1}).
 \item  $\hat{\alpha}_{(m,n)}=\lambda a$ and $\beta=\overline{\lambda}\hat{b}_{(m,l)}$, for some $\lambda\in\hat{\mathbb{C}}$, then
$A=A(\dfrac{1}{\lambda} \hat{\alpha}_{(m,n)},\alpha)+a_{0}I_{n\times m}$, and $B=A(b,\overline{\lambda}\hat{b}_{(m,l)})+b_{0}I_{m\times l}$.
 In this case, one can show that, if $\lambda\in\mathbb{C}^{\ast},$ then 
$$
A=\left[\begin{array}{ccccccc}
a_{0} & \overline{\alpha}_{1} & \cdots  &   \overline{\alpha}_{m-n} &\lambda a_{n-1} &\cdots & \lambda a_{1}\\
a_{1}  & \ddots  &   & \\
 \vdots & &  & \\
 a_{n-1}& &   &  
\end{array}\right],$$

and $$B=\left[\begin{array}{ccccccc}
b_{0} & \lambda b_{m-1} & \cdots  &   \lambda b_{m-l+1} \\
b_{1} & \ddots  &   & \\
 \vdots & &  & \\
b_{m-1} & &   &  
\end{array}\right].$$

If $\lambda =0$, then $\beta=0$ and $\hat{\alpha}_{(m,n)}=0$, this implies that $B=A(b,0)+b_{0}I_{m\times l}$ and
\begin{equation}\label{MATRIX A3}
A =\left[\begin{array}{ccccccc}
a_{0} & \overline{\alpha}_{1} & \cdots  &   \overline{\alpha}_{m-n} & 0 &\cdots & 0\\
a_{1}  & \ddots  &   & \\
 \vdots & &  & \\
 a_{n-1}& &   &  
\end{array}\right]. 
\end{equation}
If $\lambda=\infty$, then $a=0$ and $\hat{b}_{(m,l)}=0$, this implies that $A=A(0,\alpha)+a_{0}I_{n\times m}$ and
 
$$B=\left[\begin{array}{ccccccc}
b_{0} & \overline{\beta}_{1} & \cdots  &\overline{\beta}_{m-1}  \\
b_{1} & \ddots  &   & \\
 \vdots & &  & \\
  b_{m-l} & &  & \\
    0 & &  & \\
 \vdots & &  & \\
 0 & &   & 
\end{array}\right].$$
\end{enumerate}

\end{proof}

\begin{example} The method given in Theorem \ref{coro 3.3} for constructing the asymmetric Toeplitz matrices $A$ and $B$  is illustrated by the following example.

Let $A\in\mathcal{T}(4,5)$ and $B\in\mathcal{T}(5,3)$, where
\begin{center}
$A=\left[\begin{array}{ccccc}
a &  b & \lambda & 2\lambda & 3\lambda  \\
3 &  a & b & \lambda & 2\lambda  \\
2 &  3 & a & b & \lambda  \\
1 &  2 & 3 & a & b  
\end{array}\right] $\qquad
$B=\left[\begin{array}{ccc}
c &  5\lambda & 4\lambda  \\
d &  c & 5\lambda   \\
e &  d & c   \\
4 &  e & d\\
  5 &  4 & e\\ 
\end{array}\right] ,$
\end{center}
for some $a,b,c,d,e,\lambda\in\mathbb{C},$ then we have

 $AB=\left[\begin{array}{ccc}
ac+bd+e\lambda +23\lambda &  bc+\lambda (5a+d+2e+12) & \lambda (4a+5b+c+2d+3e)  \\
3c+ad+eb+14\lambda &  ac+bd+e\lambda +23\lambda & bc+\lambda (5a+d+2e+12)  \\
2c+3d+ea+4b+5\lambda &  3c+ad+eb+14\lambda & ac+bd+e\lambda +23\lambda   \\
c+2d+3e+4a+5b &  2c+3d+ea+4b+5\lambda  & 3c+ad+eb+14\lambda  
\end{array}\right].$ 
It is clear that $AB\in\mathcal{T}(4,3).$
\end{example}

\begin{theorem}\label{coro 3.4}
Let $A\in\mathcal{T}(n,m)(A\neq 0)$ and $B\in\mathcal{T}(m,l)(B\neq 0)$. If$\hspace*{2mm}m<\min (n,l)$, and $km< n\leq (k+1)m $, $k'm< l\leq (k'+1)m $, where $k$ and $k'$ are positive integers, then $AB\in\mathcal{T}(n,l)$ if and only if $A$ has a form
\begin{equation}\label{matrix a}
\left[\begin{array}{ccccccc}
a_{0} & \dfrac{1}{\lambda}a_{m-1} & \cdots  & & &  & \dfrac{1}{\lambda}a_{1} \\
a_{1} & \ddots   & & & & & \\
 \vdots & &  &  & & & \\
a_{m-1}& &  &  & &  & \\
\lambda a_{0}& & &  &  &  & \\
\vdots & &  & & & & \\
\lambda a_{m-1}& & &  &  &  & \\
\vdots & &  & & & & \\
\lambda^{k} a_{0}&   & & & &  & \\
\vdots & &  & & & & \\
\lambda^{k} a_{n-km-1}& & & &  &  & 
\end{array}\right] 
\end{equation}
and $B$ has a form 
{\footnotesize
\begin{equation}\label{matrix b}
\left[\begin{array}{ccccccccccc}
b_{0} &  \overline{\beta}_{1} &  \cdots & \overline{\beta}_{m-1} & \dfrac{1}{\lambda}b_{0} & \cdots & \dfrac{1}{\lambda}\overline{\beta}_{m-1} & \cdots & \dfrac{1}{\lambda^{k'}}b_{0} &  \cdots & \dfrac{1}{\lambda^{k'}}\overline{\beta}_{l-k'm-1} \\ 
\lambda \overline{\beta}_{m-1} & & \ddots & & & & & & &    \\
\vdots & & & & & & & & & & \\
\lambda \overline{\beta}_{1} & & & & & & & & & & 
\end{array}\right].
\end{equation}}
\end{theorem}

\begin{proof}
 If $m <\min(n,l)$. By Lemma \ref{Lemma 2.4} and Lemma \ref{Lemma 2.5}, we have 
$$\Delta(AB)=a\otimes\beta-(\hat{\alpha}_{(m,n)}+a_{0}e_{m})\otimes (\hat{b}_{(m,l)}+\overline{b}_{0}\zeta_{m})+\gamma_{1}\otimes \zeta_{0}+e_{0}\otimes \gamma_{2},$$
with $\gamma_{1}=A_{0}b+a_{0}b_{(m,n)}^{\sharp}+ab_{0}+a_{0}b_{0}e_{0}$, $\gamma_{2}=S_{l}B_{0}^{\ast}S_{m}^{\ast}\alpha +\overline{a}_{0}\beta +\overline{b}_{0}\alpha_{(m,l)}^{\sharp}$, $\hat{\alpha}_{(m,n)}+a_{0}e_{m}=(0,\overline{\alpha}_{m-1},...,\overline{\alpha}_{1},a_{0},a_{1},...,a_{n-m-1})^{T}$, and $\hat{b}_{(m,l)}+\overline{b}_{0}\zeta_{m}=(0,\overline{b}_{m-1},...,\overline{b}_{1},\overline{b}_{0},\beta_{1},...,\beta_{l-m-1})^{T}$.

 Since $a\otimes\beta$ and $(\hat{\alpha}_{(m,n)}+a_{0}e_{m})\otimes (\hat{b}_{(m,l)}+\overline{b}_{0}\zeta_{m})$ are $n\times l$ matrices, where the first row and the first column are zero, then $AB\in\mathcal{T}(n,l)$ if and only if $$a\otimes\beta-(\hat{\alpha}_{(m,n)}+a_{0}e_{m})\otimes (\hat{b}_{(m,l)}+\overline{b}_{0}\zeta_{m})=0$$
It is clear that if $a=\hat{\alpha}_{(m,n)}+a_{0}e_{m}=0$, then $A=0$, and if $\beta =\hat{b}_{(m,l)}+\overline{b}_{0}\zeta_{m} =0$, then $B=0$.

Also, by Lemma \ref{L1}, we conclude that $AB\in\mathcal{T}(n,l)$ if and only if $a=\lambda (\hat{\alpha}_{(m,n)}+a_{0}e_{m})$ and $\hat{b}_{(m,l)}+\overline{b}_{0}\zeta_{m}=\overline{\lambda} \beta$, for some $\lambda\in\hat{\mathbb{C}}$. In this case, if $\lambda\in\mathbb{C}^{\ast}$, and $km< n\leq (k+1)m $, $k'm< l\leq (k'+1)m $, where $k$ and $k'$ are positive integers, then $A$ has a form
(\ref{matrix a}) and $B$ has a form (\ref{matrix b}). And if $\lambda=0$, this implies that $a=0$ and $\hat{b}_{(m,l)}+\overline{b}_{0}\zeta_{m}=0$, then $A=A(0,\alpha)+a_{0}I_{n\times m}$ and 

\begin{equation}\label{MATRIX B3}
B=\left[\begin{array}{ccccccc}
0 & 0 & \cdots  & 0 & \overline{\beta}_{l-m} &\cdots & \overline{\beta}_{l-1} \\
  \vdots & \ddots  &   & \\
  & &  & \\
0 & &   & 
\end{array}\right].
\end{equation}

And if $\lambda=\infty$, this implies that $\beta=0$ and $\hat{\alpha}_{(m,n)}+a_{0}e_{m}=0$, then 
\begin{equation}\label{MATRIX A5}
A=\left[\begin{array}{ccccccc}
0 & 0 & \cdots  & 0 \\
  \vdots & \ddots    & &\\
 0 & & &  \\
a_{n-m} & &   & \\
  \vdots &   &   & \\
a_{n-1} &   &   & 
\end{array}\right],
\end{equation}
 and $B=A(b,0)+b_{0}I_{m\times l}$.

\end{proof}
\begin{theorem}\label{th3}
Let $A\in\mathcal{T}(n,m)(A\neq 0)$ and $B\in\mathcal{T}(m,l)(B\neq 0)$. If$\hspace*{2mm}n\leq m < l$, then $AB\in\mathcal{T}(n,l)$ if and only if one of the following conditions is met
\begin{enumerate}
\item
\begin{equation}\label{MATRIX A2}
A =\left[\begin{array}{ccccccc}
a_{0} & \overline{\alpha}_{1} & \cdots  &   \overline{\alpha}_{m-n} & 0 &\cdots & 0\\
0  & \ddots  &   & \\
 \vdots & &  & \\
 0 & &   &  
\end{array}\right] .
\end{equation}
\item
$A=A(\lambda \hat{\alpha}_{(m,n)},\alpha)+a_{0}I_{n\times m}$ and $B$ has a form (\ref{matrix b}), for some $\lambda\in\hat{\mathbb{C}}$.
\end{enumerate}

\end{theorem}
\begin{proof}
If $n \leq m < l$. By Lemma \ref{Lemma 2.4} and Lemma \ref{Lemma 2.5}, we have 
$$\Delta(AB)=a\otimes\beta-\hat{\alpha}_{(m,n)}\otimes (\hat{b}_{(m,l)}+\overline{b}_{0}\zeta_{m})+\gamma_{1}\otimes \zeta_{0}+e_{0}\otimes \gamma_{2},$$
with $\gamma_{1}=A_{0}b+a_{0}b_{(m,n)}^{\sharp}+ab_{0}+a_{0}b_{0}e_{0}$ and $\gamma_{2}=S_{l}B_{0}^{\ast}S_{m}^{\ast}\alpha +\overline{a}_{0}\beta +\overline{b}_{0}\alpha_{(m,l)}^{\sharp}$.

 Since $a\otimes\beta$ and $\hat{\alpha}_{(m,n)}\otimes (\hat{b}_{(m,l)}+\overline{b}_{0}\zeta_{m})$ are $n\times l$ matrices, where the first row and the first column are zero, then $AB\in\mathcal{T}(n,l)$ if and only if 
 $$a\otimes\beta-\hat{\alpha}_{(m,n)}\otimes (\hat{b}_{(m,l)}+\overline{b}_{0}\zeta_{m})=0.$$
It is clear that if $\beta =\hat{b}_{(m,l)}+\overline{b}_{0}\zeta_{m} =0$, then $B=0$. 

Also, by Lemma \ref{L1}, we conclude that $AB\in\mathcal{T}(n,l)$ if and only if one of the following conditions is met
\begin{enumerate}
\item
$a=\hat{\alpha}_{(m,n)}=0$, then $A$ has a form (\ref{MATRIX A2}).
\item
$a=\lambda \hat{\alpha}_{(m,n)}$ and $\hat{b}_{(m,l)}+\overline{b}_{0}\zeta_{m}=\overline{\lambda} \beta$, for some $\lambda\in\hat{\mathbb{C}}$. In this case, if $\lambda\in\mathbb{C}^{\ast}$, and $k'm< l\leq (k'+1)m $, where $k'$ is a positive integer, then $A=A(\lambda \hat{\alpha}_{(m,n)},\alpha)+a_{0}I_{n\times m}$ and $B$ has a form (\ref{matrix b}). And if $\lambda=0$, this implies that $a=0$ and $\hat{b}_{(m,l)}+\overline{b}_{0}\zeta_{m}=0$, then $A=A(0,\alpha)+a_{0}I_{n\times m}$ and $B$ has a form (\ref{MATRIX B3}). And if $\lambda= \infty$, this implies that $\beta=0$ and $\hat{\alpha}_{(m,n)}=0$, then $B=A(b,0)+b_{0}I_{m\times l}$ and $A$ has a form (\ref{MATRIX A3}). 
\end{enumerate}

\end{proof}

\begin{theorem}\label{th4}
Let $A\in\mathcal{T}(n,m)(A\neq 0)$ and $B\in\mathcal{T}(m,l)(B\neq 0)$. If$\hspace*{2mm}l\leq m < n$, then $AB\in\mathcal{T}(n,l)$ if and only if one of the following conditions is met
\begin{enumerate}
\item
\begin{equation}\label{MATRIX B4}
B =\left[\begin{array}{ccccccc}
b_{0} & 0 & \cdots  &  & 0\\
b_{1}  & \ddots  &   & \\
 \vdots & &  & \\
 b_{m-l} & &   &  \\
 0 & &   &  \\
 \vdots & &   & \\
  0 & &   &  
\end{array}\right] .
\end{equation}
\item
$A$ has a form (\ref{matrix a}) and $B=A(b,\dfrac{1}{\overline{\lambda}}\hat{b}_{(m,l)})+b_{0}I_{m\times l}$, for some $\lambda\in\hat{\mathbb{C}}$.
\end{enumerate}

\end{theorem}
\begin{proof}
If $l \leq m < n$. By Lemma \ref{Lemma 2.4} and Lemma \ref{Lemma 2.5}, we have 
$$\Delta(AB)=a\otimes\beta-(\hat{\alpha}_{(m,n)}+a_{0}e_{m})\otimes \hat{b}_{(m,l)}+\gamma_{1}\otimes \zeta_{0}+e_{0}\otimes \gamma_{2}.$$
with $\gamma_{1}=A_{0}b+a_{0}b_{(m,n)}^{\sharp}+ab_{0}+a_{0}b_{0}e_{0}$ and $\gamma_{2}=S_{l}B_{0}^{\ast}S_{m}^{\ast}\alpha +\overline{a}_{0}\beta +\overline{b}_{0}\alpha_{(m,l)}^{\sharp}$.

Since $a\otimes\beta$ and $(\hat{\alpha}_{(m,n)}+a_{0}e_{m})\otimes \hat{b}_{(m,l)}$ are $n\times l$ matrices, where the first row and the first column are zero, then $AB\in\mathcal{T}(n,l)$ if and only if 
$$a\otimes\beta-(\hat{\alpha}_{(m,n)}+a_{0}e_{m})\otimes \hat{b}_{(m,l)}=0.$$
It is clear that if $a=\hat{\alpha}_{(m,n)}+a_{0}e_{m}=0$, then $A=0$. 

Also, by Lemma \ref{L1}, we conclude that $AB\in\mathcal{T}(n,l)$ if and only if one of the following conditions is met
\begin{enumerate}
\item
$\beta=\hat{b}_{(m,l)}=0$, then $B$ has a form (\ref{MATRIX B4}).
\item
$a=\lambda (\hat{\alpha}_{(m,n)}+a_{0}e_{m})$ and $\hat{b}_{(m,l)}=\overline{\lambda} \beta$, for some $\lambda\in\hat{\mathbb{C}}$. In this case, if $\lambda\in\mathbb{C}^{\ast}$, and $km< n\leq (k+1)m $, where $k$ is a positive integer, then $A$ has a form
(\ref{matrix a}) and $B=A(b,\dfrac{1}{\overline{\lambda}}\hat{b}_{(m,l)})+b_{0}I_{m\times l}$. And if $\lambda=0$, this implies that $a=0$ and $\hat{b}_{(m,l)}=0$, then $A=A(0,\alpha)+a_{0}I_{n\times m}$ and 
$$B =\left[\begin{array}{cccccc}
b_{0} & \overline{\beta}_{1} & \cdots  &  & \overline{\beta}_{l-1}\\
b_{1}  & \ddots  &   & \\
 \vdots & &  & \\
 b_{m-l} & &   &  \\
 0 & &   &  \\
 \vdots & &   & \\
   0 & &   &  
\end{array}\right].$$

And if $\lambda=\infty $, this implies that $\beta=0$ and $\hat{\alpha}_{(m,n)}+a_{0}e_{m}=0$, then $B=A(b,0)+b_{0}I_{m\times l}$ and $A$ has a form (\ref{MATRIX A5}).
\end{enumerate}

\end{proof}
Using the same arguments of Theorems \ref{coro 3.3}, \ref{coro 3.4}, \ref{th3}, and \ref{th4}, we can check the following consequence.

\begin{corollary}
Let $H_{1}\in\mathcal{H}(n,m)$ and $H_{2}\in\mathcal{H}(m,l)$. Then there are two matrices  $A\in\mathcal{T}(n,m)$ and $B\in\mathcal{T}(m,l)$, such that $AP_{m}=H_{1}$ and $P_{m}B= H_{2}$. Hence, $H_{1}H_{2}\in\mathcal{T}(n,l)$ if and only if $AB\in\mathcal{T}(n,l)$.
\end{corollary}
\begin{proof} Since $P_{m}^{2}=I_{m}$, we have 
$$H_{1}H_{2}=AP_{m}P_{m}B=AP_{m}^{2}B=AB,$$
this ends the proof.
\end{proof}
\begin{corollary}
Let $H\in \mathcal{H}(n,m)$ such that $H=P_{n}A$, where $A\in\mathcal{T}(n,m)$, and let $B\in\mathcal{T}(m,l)$. Then $HB\in\mathcal{H}(n,l)$ if and only if $AB\in\mathcal{T}(n,l)$.
\end{corollary}
\begin{proof} Suppose that $AB\in\mathcal{T}(n,l)$. We have $HB=P_{n}AB$, thus $HB\in\mathcal{H}(n,l)$. 

\end{proof}

\section{ isometry of asymmetric Toeplitz (Hankel) matrix}
In this section, we deal with the isometric asymmetric Toeplitz matrices.
\begin{lemma}\label{lemma12}
Let $A\in\mathcal{T}(n,m)$. 
\begin{enumerate}
\item If $m\leq n$, then
\begin{align*}
\Delta(A^{\ast}A-I_{m})&=\alpha\otimes\alpha-\hat{a}_{(n,m)}\otimes \hat{a}_{(n,m)}+\Big(A_{0}^{\ast}a+\overline{a}_{0}a_{(n,m)}^{\sharp}+a_{0}\alpha+\dfrac{\vert a_{0}\vert^{2}-\Vert a \Vert^{2} -1}{2}\varepsilon_{0}\Big)\otimes \varepsilon_{0}\\
&+\varepsilon_{0}\otimes \Big(A_{0}^{\ast}a+\overline{a}_{0}a_{(n,m)}^{\sharp}+a_{0}\alpha+\dfrac{\vert a_{0}\vert^{2}-\Vert a \Vert^{2} -1}{2}\varepsilon_{0}\Big),
\end{align*}
where $\Vert a \Vert^{2}=\langle a,a\rangle =\sum\limits_{i=1}^{n-1}\vert a_{i} \vert^{2}$.
\item If $n< m$, then
\begin{align*}
\Delta(A^{\ast}A-I_{m})&=\alpha\otimes\alpha-(\hat{a}_{(n,m)}+\overline{a}_{0}\varepsilon_{n})\otimes (\hat{a}_{(n,m)}+\overline{a}_{0}\varepsilon_{n})+\Big(A_{0}^{\ast}a+\overline{a}_{0}a_{(n,m)}^{\sharp}+a_{0}\alpha\\
 &+\dfrac{\vert a_{0}\vert^{2}-\Vert a \Vert^{2} -1}{2}\varepsilon_{0}\Big)\otimes \varepsilon_{0}
+\varepsilon_{0}\otimes \Big(A_{0}^{\ast}a+\overline{a}_{0}a_{(n,m)}^{\sharp}+a_{0}\alpha+\dfrac{\vert a_{0}\vert^{2}-\Vert a \Vert^{2} -1}{2}\varepsilon_{0}\Big).
\end{align*}
\end{enumerate}
\end{lemma}
\begin{proof}
\begin{enumerate}
\item
If $m\leq n$. By Lemma \ref{Lemma 2.4} and Lemma \ref{Lemma 2.5}, we have
\begin{align}\label{eq:01}
\Delta(A^{\ast}A-I_{m})&=\alpha\otimes\alpha-\hat{a}_{(n,m)}\otimes \hat{a}_{(n,m)}+(A_{0}^{\ast}a+\overline{a}_{0}a_{(n,m)}^{\sharp}+a_{0}\alpha+\vert a_{0}\vert^{2} \varepsilon_{0})\otimes \varepsilon_{0}\nonumber\\
&+\varepsilon_{0}\otimes (S_{m}A_{0}^{\ast}S_{n}^{\ast}a +\overline{a}_{0}a_{(n,m)}^{\sharp}+ a_{0}\alpha)-\varepsilon_{0}\otimes\varepsilon_{0}.
\end{align}
On the other hand, we have
$$\Delta A_{0}^{\ast}=A_{0}^{\ast}-S_{m}A_{0}^{\ast}S_{n}^{\ast}=\alpha\otimes e_{0}+\varepsilon_{0}\otimes a,$$
then 
$$S_{m}A_{0}^{\ast}S_{n}^{\ast}=A_{0}^{\ast}-\alpha\otimes e_{0}-\varepsilon_{0}\otimes a,$$
since $(\alpha \otimes e_{0})a=0$, then $$S_{m}A_{0}^{\ast}S_{m}^{\ast}a=A_{0}^{\ast}a-\Vert a \Vert^{2}\varepsilon_{0},$$
 which, when substituted into (\ref{eq:01}), gives us 
\begin{align*}
\Delta(A^{\ast}A-I_{m})&=\alpha\otimes\alpha-\hat{a}_{(n,m)}\otimes \hat{a}_{(n,m)}+(A_{0}^{\ast}a+\overline{a}_{0}a_{(n,m)}^{\sharp}+a_{0}\alpha)\otimes \varepsilon_{0}\\
&+\varepsilon_{0}\otimes (A_{0}^{\ast}a+\overline{a}_{0}a_{(n,m)}^{\sharp}+a_{0}\alpha )
+(\vert a_{0}\vert^{2}-\Vert a\Vert^{2} -1)\varepsilon_{0}\otimes\varepsilon_{0},
\end{align*}
then
\begin{align*}
\Delta(A^{\ast}A-I_{m})&=\alpha\otimes\alpha-\hat{a}_{(n,m)}\otimes \hat{a}_{(n,m)}+\Big(A_{0}^{\ast}a+\overline{a}_{0}a_{(n,m)}^{\sharp}+a_{0}\alpha+\dfrac{\vert a_{0}\vert^{2}-\Vert a \Vert^{2} -1}{2}\varepsilon_{0}\Big)\otimes \varepsilon_{0}\\
&+\varepsilon_{0}\otimes \Big(A_{0}^{\ast}a+\overline{a}_{0}a_{(n,m)}^{\sharp}+a_{0}\alpha+\dfrac{\vert a_{0}\vert^{2}-\Vert a \Vert^{2} -1}{2}\varepsilon_{0}\Big).
\end{align*}
\item
The proof is the same as proof (1).
\end{enumerate}

\end{proof}

\begin{theorem}
Let $A\in\mathcal{T}(n,m)(A\neq 0)$. If $m\leq n$, then $A$ is an isometric matrix if and only if
 $A_{0}^{\ast}a+\overline{a}_{0}a_{(n,m)}^{\sharp}+a_{0}\alpha+\dfrac{\vert a_{0}\vert^{2}-\Vert a \Vert^{2} -1}{2}\varepsilon_{0}=0$ and one of the following two cases holds
\begin{enumerate}
\item
$A$ has a form (\ref{MATRIX B1}).
\item
$A=A(a,\lambda\hat{a}_{(m,n)})+a_{0}I_{n\times m}$ for some $\lambda\in\mathbb{C}$, such that $\vert \lambda\vert =1$.
\end{enumerate}
\end{theorem}
\begin{proof}
If $m\leq n$. $A$ is an isometric matrix if and only if $\Delta (A^{\ast}A-I_{m})=0$.

By Lemma \ref{lemma12}, and since $\alpha\otimes\alpha$ and $\hat{a}_{(n,m)}\otimes \hat{a}_{(n,m)}$ are $m\times m$ matrices, where the first row and the first column are zero, then $\Delta (A^{\ast}A-I_{m})=0$ if and only if
\begin{equation}\label{eq.Isometry}
\alpha\otimes\alpha-\hat{a}_{(n,m)}\otimes \hat{a}_{(n,m)}=0
\end{equation}
and
\begin{center}
$A_{0}^{\ast}a+\overline{a}_{0}a_{(n,m)}^{\sharp}+a_{0}\alpha+\dfrac{\vert a_{0}\vert^{2}-\Vert a \Vert^{2} -1}{2}\varepsilon_{0}=0.$
\end{center}
By Lemma \ref{L1}, the equation (\ref{eq.Isometry}) is satisfied if and only if one of the following two cases holds
\begin{enumerate}
\item $a=\hat{\alpha}_{(m,n)}=0$, this implies that $A$ has a form (\ref{MATRIX B1}).
\item $\alpha=\lambda \hat{a}_{(m,n)}$ and $\hat{a}_{(m,n)}=\overline{\lambda} \alpha $, for some $\lambda \in\mathbb{C}$, then $\alpha=\lambda (\overline{\lambda} \alpha )=\vert\lambda\vert^{2}\alpha$, which implies that $\vert\lambda\vert =1$.
\end{enumerate}
Hence, we complete the proof.
\end{proof}
\begin{theorem}
Let $A\in\mathcal{T}(n,m)(A\neq 0)$. If $n<m$, then $A$ is an isometric matrix if and only if $A$ has a form (\ref{matrix b}) with $\vert \lambda\vert =1$ and $A_{0}^{\ast}a+\overline{a}_{0}a_{(n,m)}^{\sharp}+a_{0}\alpha+\dfrac{\vert a_{0}\vert^{2}-\Vert a \Vert^{2} -1}{2}\varepsilon_{0}=0$.
\end{theorem}
\begin{proof}
If $n<m$. $A$ is an isometric matrix if and only if $\Delta (A^{\ast}A-I_{m})=0$.

By Lemma \ref{lemma12}, and since $\alpha\otimes\alpha$ and $(\hat{a}_{(n,m)}+\overline{a}_{0}\varepsilon_{n})\otimes (\hat{a}_{(n,m)}+\overline{a}_{0}\varepsilon_{n})$ are $m\times m$ matrices, where the first row and the first column are zero, then $\Delta (A^{\ast}A-I_{m})=0$ if and only if
\begin{equation}\label{eq.Isometry2}
\alpha\otimes\alpha-(\hat{a}_{(n,m)}+\overline{a}_{0}\varepsilon_{n})\otimes (\hat{a}_{(n,m)}+\overline{a}_{0}\varepsilon_{n})=0
\end{equation}
and
\begin{center}
$A_{0}^{\ast}a+\overline{a}_{0}a_{(n,m)}^{\sharp}+a_{0}\alpha+\dfrac{\vert a_{0}\vert^{2}-\Vert a \Vert^{2} -1}{2}\varepsilon_{0}=0.$
\end{center}
By Lemma \ref{L1}, the equation (\ref{eq.Isometry2}) is satisfied if and only if $\alpha=\lambda (\hat{a}_{(n,m)}+\overline{a}_{0}\varepsilon_{n})$ and $\hat{a}_{(n,m)}+\overline{a}_{0}\varepsilon_{n}=\overline{\lambda} \alpha $, for some $\lambda \in\mathbb{C}$, then $\alpha=\lambda (\overline{\lambda} \alpha )=\vert\lambda\vert^{2}\alpha$, which implies that $\vert\lambda\vert =1$. In which case $A$ has a form (\ref{matrix b}).

\end{proof}
\begin{remark}
Let $A\in\mathcal{T}(n,m)$. If $A$ is an isometric matrix, then $\sum\limits_{i=0}^{n-1}\vert a_{i} \vert^{2}=1$.
\end{remark}
\begin{proof}
Assume that $A$ is isometric. By Lemma \ref{lemma12}, we have 
\begin{equation}\label{eq03}
A_{0}^{\ast}a+\overline{a}_{0}a_{(n,m)}^{\sharp}+a_{0}\alpha +(\vert a_{0}\vert^{2}-1)\varepsilon_{0}=0
\end{equation}
and
\begin{equation}\label{eq04}
S_{m}A_{0}^{\ast}S_{n}^{\ast}a +\overline{a}_{0}a_{(n,m)}^{\sharp}+ a_{0}\alpha=0.
\end{equation}
By subtracting between (\ref{eq03}) and (\ref{eq04}), we get
$$(\Delta A_{0}^{\ast})a+(\vert a_{0}\vert^{2}-1)\varepsilon_{0}=0,$$
then
$$\Vert a \Vert^{2} +\vert a_{0}\vert^{2}-1=0,$$
hence
 $\sum\limits_{i=0}^{n-1}\vert a_{i} \vert^{2}=1$.

\end{proof}
\begin{example}
Let 
$A=\left[\begin{array}{cc}
\dfrac{1}{2}i &   \dfrac{1}{4}-\dfrac{\sqrt{7}}{4}i\\
\dfrac{1}{2} &  \dfrac{1}{2}i \\
\dfrac{\sqrt{7}}{4}+\dfrac{1}{4}i &  \dfrac{1}{2}
\end{array}\right]\in\mathcal{T}(3,2)$,
it is easy to verify that $A$ is an isometric matrix.
\end{example}

\begin{corollary}
Let $H\in\mathcal{H}(n,m)$. Then 
$H$ is an isometric matrix if and only if the asymmetric Toeplitz matrix $P_{n}H$ is isometric.
\end{corollary}
\begin{proof}
Suppose that $P_{n}H$ is an isometric matrix. Since $P_{n}^{2}=I_{n}$, then we have

$$H^{\ast}H=H^{\ast}P_{n}^{2}H=H^{\ast}P_{n}^{\ast}P_{n}H=(P_{n}H)^{\ast}(P_{n}H)=I_{m}
,$$

thus $H$ is isometric.
\end{proof}
\section*{Acknowledgments} 
This research work
is supported by the General Direction of Scientific Research and Technological
Development (DGRSDT), Algeria.
\bibliographystyle{amsplain}

\end{document}